\documentclass{amsproc}

\newtheorem{theorem}{Theorem}[section]
\newtheorem{lemma}[theorem]{Lemma}
\newtheorem{proposition}[theorem]{Proposition}

\newtheorem{corollary}[theorem]{Corollary}
\theoremstyle{definition}
\newtheorem{definition}[theorem]{Definition}
\newtheorem{example}[theorem]{Example}

\theoremstyle{remark}
\newtheorem{remark}[theorem]{Remark}

\numberwithin{equation}{section}

\begin{document}

\title {ON THE ZARISKI TOPOLOGY of $\Omega$-GROUPS}

\author{R. Lipyanski}
\address{Department of Mathematics, Ben Gurion University, Beer Sheva 84105, Israel}

\curraddr{Department of Mathematics, Ben Gurion University, Beer Sheva 84105, Israel}
\email{lipyansk@math.bgu.ac.il}



\dedicatory{This article is dedicated to my teacher Prof. B. Plotkin on his 90th anniversary.}


\begin{abstract}
A number of geometric properties of $\Omega$-groups from a given variety of $\Omega$-groups can be characterized using the notions of domain and equational domain. An $\Omega$-group $H$ of a variety $\Theta$ is an equational domain in $\Theta$ if the union of algebraic varieties over $H$ is an algebraic variety.  We give necessary and sufficient conditions for an $\Omega$-group $H$ in $\Theta$ to be an equational domain in this variety.
\end{abstract}

\maketitle


Let $F=F(X)$ be a finitely generated by $X$ free $\Omega$-group in a variety of $\Omega$-groups $\Theta$ and $H$ be an $\Omega$-group in $\Theta$. One of the important questions in the algebraic geometry of  varieties of $\Omega$-groups is whether it is possible to equip the space of points ${\rm Hom}(F,H)$ with the Zariski topology, whose closed sets are precisely algebraic sets. If this is so, then the $\Omega$-group $H$ is called the equational domain (or \textit{stable} in the terminology of the authors of the papers \cite {P}, \cite {BPP}) in the variety $\Theta$ (see \cite {BMR}). The same problem arises for a variety $\Theta(G)$ of $\Omega$-groups with the given $\Omega$-group of constants $G$.

  An important role in the study of equational domains in varieties of $\Omega$-groups is played by the notion of domain.
 Necessary and sufficient conditions for linear algebras over an algebra of constants and for groups over a group of constants to be equational domains in terms of domains are given in \cite{BMR}, \cite {BPP}, \cite{DMR} and \cite {P}. Here we continue the study of equational domains in varieties of $\Omega$-groups (without of the $\Omega$-group of constants). We give necessary and sufficient conditions for an $\Omega$-group $H$ in $\Theta$ to be an equational domain in the variety $\Theta$.

 The results presented in this paper were partially announced earlier in \cite {L}.

 \section{Preliminaries}
\subsection {$\Omega$-groups}
 Now we give the basic facts about $\Omega$-groups (see \cite {H}).
 \begin{definition}
 An $\Omega$-group (multioperator group) $G$ is an additive group (not necessarily commutative) with some additional signature $\Omega$, such that for every $\omega\in \Omega$ of the arity $n(\omega)=n>0$, the condition $\underbrace{00\dots0}_{n}\omega=0$ should be fulfilled.
 \end{definition}

 A group is an $\Omega$-group with the empty set of operations $\Omega$; in rings the set $\Omega$ consists of a single multiplication; in Lie algebras over a commutative associative ring $K$ with unit the set $\Omega$ consists of the Lie bracket and all elements of $K$ belonging to the set $\Omega$.

 It is clear that the class of all $\Omega$-groups forms the variety $\Upsilon$. Let $F_\Upsilon=F_\Upsilon(X)$ be a finitely generated by $X$ free $\Omega$-group in the variety $\Upsilon$. We use the functional notion $f(x_1,\dots,x_n), x_i\in X$, for words in $F_\Upsilon(X)$. For brevity we shall use symbols $\bar x,\bar a,$ to denote finite ordered sets $(x_1,x_2,\dots,x_n), (a_1,a_2,\dots,a_m), x_i\in X, a_i\in G$, and write $f(\bar x)$ for $f(x_1,x_2,\dots,x_n)$, $f(\bar a,\bar b)$ for $f(a_1,a_2,\dots,a_m, b_1,\dots, b_r)$. If $\bar x, \bar y$ are ordered sets $(x_1,x_2,\dots,x_n), (y_1,y_2,\dots,y_n)$ with the same number of elements, we shall denote the set
  $x_1+y_1,x_2+y_2,\dots,x_n+y_n$ by $\bar x +\bar y$. We shall also write $\bar a\in G$ when we mean $a_1,\dots, a_m\in G$.

  Let $f(\bar x, \bar y)$ be a word in two disjoint sets $X$ and $Y$ of variables $\bar x$ and $\bar y$, respectively.
  We shall say that the word $f(\bar x, \bar y)$ is a \textit{commutator word in $\bar x$ and $\bar y$} if $f(\bar x, \bar 0)=f(\bar 0, \bar y)=0$. The set of all such words will be denoted by $[X,Y]$.

  Let $A$ and $B$ be two subsets of an $\Omega$-group $G$. The set of all elements $f(\bar a,\bar b)$ , where
 $f(\bar x, \bar y)\in [X,Y]$ is called the \textit{commutator group of $A$ and $B$} and is denoted by $[A,B]$. Note that if $A$ and $B$ are $\Omega$-subgroups, then $[A,B]$ is also an $\Omega$-subgroup of $G$.

 \begin{lemma}[\cite{H}]\label{com}
 If $[A,B]=0$, then
 $$
f(\bar a,\bar b)=f(\bar a,\bar 0)+f(\bar 0,\bar b)=f(\bar 0,\bar b)+f(\bar a,\bar 0)
 $$
 for all words $f(\bar x, \bar y)$  and all $\bar a\in A$ and $\bar b\in B$.
 \end{lemma}
  Let $A$ and $B$ be defined as above and $\bar a=(a_1,a_2,\dots,a_n), a_i\in A$ and $\bar b=(b_1,b_2,\dots,b_n), b_i\in B$. Then the element
 $$
 [\bar a;\bar b;\omega]=-\bar a\omega-\bar b\omega+(\bar a+\bar b)\omega, \omega\in \Omega
 $$
 is called the \textit{$\omega$-commutator of $\bar a$ and $\bar b$}.
 \begin{definition}
 Let $G$ be an $\Omega$-group. A subset $U$of $G$ is called an ideal in $G$  if the following conditions are fulfilled:
 \begin{enumerate}
\item $U$ is closed with respect to all $\omega\in \Omega$.
\item $U$ is a normal subgroup in additive group $G$.
\item The $\omega$-commutator $[\bar a;\bar b;\omega]$ belongs to $U$ if $\bar a\in U$, $\omega\in \Omega$ and $\bar b\in G$.
\end{enumerate}
 \end{definition}
 Let $A$ and $B$  be two $\Omega$-subgroups of $G$ and let $\{A,B\}$ be an $\Omega$-subgroup of $G$ generated by $A$ and $B$.  The commutator group $[A,B]$ can be characterized as follows.
  \begin{proposition}[\cite{H}]
 The commutator group $[A,B]$ is the ideal in $\{A,B\}$ generated by all commutators of the kind $[a,b], a\in A, b\in B,$ and all $\omega$-commutators $[\bar a;\bar b;\omega]$, where $\bar a\in A$ and $\bar b\in B, \omega\in\Omega$.
 \end{proposition}
   \begin{example}
 Let $R$ be an associative ring and $U_1, U_2 $ subrings in $R$. Then it is easy to show that
 $$
 [U_1,U_2]=U_1 U_2+U_1 U_2
  $$
  In the case of Lie algebras (groups) we have the ordinary commutator subalgebra (commutator subgroup).
 \end{example}
\begin{definition}
An $\Omega$-group $G$ is called abelian if $[G,G]=0$.
 \end{definition}
Groups and Lie algebras are abelian in the usual sense, while for associative rings this notion means that the product of any two elements is zero.

Now we turn to the  property of anticommutativity for $\Omega$-groups.
\begin{definition}[\cite{BPP}]
An $\Omega$-group  $L$ is called anticommutative (or antiabelian in the terminology of the authors in \cite {BMR})  if the following conditions are fullfilled:
\begin{enumerate}
\item $L$ has no nontrivial abelian ideal.
\item  Every two nontrivial ideals $H_1$ and $H_2$ in $L$ have a nontrivial intersection.
\end{enumerate}
\end{definition}
A number of interesting properties of anticommutativity are given for groups in \cite{BMR} and for $\Omega$-groups in \cite{BPP}, \cite{P}. It is known that every non-abelian free group, free associative algebra, and non-abelian free Lie algebra are anticommutative (see \cite{BMR}, \cite{BPP}, \cite{L}).
\subsection {Domains}
Now we consider the notion of zero divisors in an $\Omega$-group $H$. For each $a\in H$, denote by $\rm{id}{\langle a \rangle}$ the ideal in $H$ generated by $a$ and $ {\langle a \rangle} $ is an $\Omega $-subgroup of $H$ generated by $a$. Let $P$ be an $\Omega$- subgroup of $H$. We denote by $\rm{id}_P{\langle a \rangle}$  the ideal in $P$ generated by $a$. In our notation, we have
 $\rm{id}_H{\langle a \rangle}=\rm{id}{\langle a \rangle}$.
\begin{definition}[\cite{P}]\label{dv}
A non-zero element $a\in H$ is called a zero divisor if for some non-zero element $b\in H$ we have
\begin{equation}
[\rm{id}{\langle a \rangle },\rm{id}{{\langle b \rangle}}]=0
\end{equation}
The  $\Omega$-group $G$ is called a \textit{domain} if $G$ is without zero divisors, i.e., for any two elements $g_1$ and $g_2$ of $G$ the following holds:
$$
[\rm{id}{\langle g_1 \rangle },\rm{id}{{\langle g_2 \rangle}}]=0\Rightarrow g_1=0\;\; \mbox{or}\;\; g_2=0,
$$
\end{definition}
\begin{example}
Let $R$ be an associative ring. In this case Definition \ref{dv} looks as follows: a non-trivial element $a$ in $A$-associative ring $R$ is a zero divisor  if there exists a non-trivial element $b\in R$ such that
$$
\rm{id}{{\langle a \rangle }}\cdot\rm{id}{{\langle b \rangle}}=\rm{id}{{\langle b \rangle }}\cdot\rm{id}{{\langle a \rangle}}=0
$$
Let $L$ be a Lie algebra over an associative commutative ring $K$ with unit. A non-trivial element $a$ in $L$ is a zero-divisor if for some non-zero element $b\in L$ is fulfilled
 $$
 [\rm{id}{{\langle a \rangle }},\rm{id}{{\langle b \rangle}}]=0,
 $$
 where $[,]$ is the Lie bracket.

Let $H$ be a group. A non-trivial element $a$ in $H$ is a zero divisor if for some non-trivial element $b\in H$ we have
\begin{equation}\label{gr}
[g_{1}^{-1}ag_{1},g_{2}^{-1}b g_2]=1
\end{equation}
for all $g_1,g_2\in H$. Here $[,]$ is the usual commutator brackets in the group $H$.
\end{example}
\begin{remark}\label{rem_1}
In \cite{BMR}) it was proved that the condition (\ref{gr}) is equivalent to the following: a non-trivial element $a$ in $H$ is a zero divisor if for some non-trivial element $b$ in $H$ we have
 \begin{equation}\label{gr_2}
[g^{-1}ag,b ]=1
\end{equation}
for all $g\in H$.
\end{remark}

\subsection {Algebraic varieties over groups}
Let $\Theta$ be a variety of $\Omega$-groups and $F=F(X)$ be a finitely generated by $X$ free group in $\Theta$. Consider an $\Omega$-group $H$ in $\Theta$. Any formula $w\equiv w^{\prime}, w, w^{\prime}\in F(X)$ can be treated  as an equation. Denote it as $w=w^{\prime}$. Every solution of this equation in $H$ is a homomorphism $\mu:F(X)\rightarrow H$ such that $w^{\mu}=w^{\prime\mu}$. It is possible to define a Galois correspondence $'$ between subsets in ${\rm Hom}(F(X),H)$ and subsets in $F(X)$. For a subset $T$ in $F(X)$ define $T'$, \textit{$H$-closure of $T$}
$$
T'=\{\mu\in{\rm Hom}(F(X),H)\;|\;T\subseteq {\rm Ker}\mu \}
$$
 On the other hand, for any subset $A\subseteq {\rm Hom}(F(X),H)$ define a set $A'$ in $F(X)$, \textit{$H$-closure of $A$}
 $$
 A'=\bigcap\limits_{\mu\in A} {\rm Ker}\mu
 $$
The set $A'$ is an ideal in $F(X)$.
 \begin{definition}
A subset $A\subseteq {\rm Hom}(F(X),H)$ is called an affine algebraic variety over $H$ if there exists a set $T$ in $F(X)$) such that $T'=A$.
\end{definition}
 If $A$ is an algebraic variety, then $A'$ is called the ideal in $F(X)$ corresponding to $A$.

  The intersection $A\cap B$ of algebraic varieties $A$ and $B$ is also an algebraic variety. The union $A\cup B$ of algebraic varieties is not necessarily an algebraic variety. If $A=T_1'$ and  $B=T_2'$, then $A\cup B\subseteq (T_1\cap T_2)'$.

 \begin{definition}[\cite {P}, \cite{DMR}]
 An $\Omega $-group $H$  is called an equational domain in  $\Theta$  if for any free $\Omega$-group $ F(X)$  and any two
 algebraic varieties $A$ and $B$ in the space ${\rm Hom}(F(X),H)$ the union $A\cup B$ is also an algebraic variety.
 \end{definition}

Equational domains play an important role in the theory of algebraic varieties. Following \cite{BPP}, denote by ${\rm Alv}_H(F)$ the set of all algebraic varieties in ${\rm Hom}(F,H)$. The set ${\rm Alv}_H(F)$ can be considered as a lattice, where the union $A\vee B$ is defined by
 $$
 A\vee B=(A\cup B)'
 $$
Denote by ${\rm Cl}_H(F)$ the set of all $H$-closed congruences in $F$. Lattice operations can be defined in a similar way in the set ${\rm Cl}_H(F)$. The lattices ${\rm Cl}_H(F)$ and ${\rm Alv}_H(F)$ are antiisomorphic. It is clear that if a $\Omega$-group $H$ is an equational domain, then the lattices ${\rm Alv}_H(F)$ and ${\rm Cl}_H(F)$ are distributive.

\section {Equational domains in a variety of $\Omega$-groups}
As before, let $\Theta$ be a variety of $\Omega$-groups.
\begin{theorem}\label{main}
An $\Omega$-group $H$ in $\Theta$ is a domain if and only if $H$ is an equational domain in $\Theta$.
\end{theorem}
\begin{proof}
Necessity has been proved by B. Plotkin (see Theorem 1 in \cite{P}). We present this proof for completeness.

Let $F(X)$ the free $\Omega$-group in $\Theta$ generated by $X$. Suppose that the $\Omega$-group $H$ is a domain. Let us take two algebraic varieties $A$ and $B$ in $V={\rm Hom}(F(X),H)$. Let $T_1$ and $T_2$ be ideals in $F(X)$ corresponding to these varieties. We check that $A\cup B=(T_1\cap T_2)'$. It is obvious that
$$
A\cup B\subseteq(T_1\cap T_2)'
$$
To check the inverse inclusion, it suffices to show that $\mu\notin A\cup B$ implies $\mu\notin (T_1\cap T_2)'$. Since
 $\mu\notin A\cup B$, we have $T_1\notin{\rm Ker}\mu$ and $T_2\notin{\rm Ker}\mu$. Hence, there exist $u\in T_1$ and $v\in T_2$ such that $u^\mu=a\neq 0$ and $v^\mu=b\neq 0$. Since $H$ is the domain,
 $[\rm{id}{\langle a \rangle },\rm{id}{{\langle b \rangle}}]\neq 0$.

  The ideal $[\rm{id}{\langle a \rangle },\rm{id}{{\langle b \rangle}}]$ is generated by $\omega$-commutators
   $$
   c=[a_1,a_2,\dots,a_n; b_1,b_2,\dots,b_n;\omega]
   $$
 and ordinary commutators $[a',b']$, where $a',a_i\in \rm{id}{\langle a \rangle }$ and $b',b_i\in \rm{id}{\langle b \rangle}, i=1,\dots,n$. Hence, there exists a nonzero $\omega$-commutator $c=[a_1,a_2,\dots,a_n; b_1,b_2,\dots,b_n;\omega]$ or a nonzero commutator $[a',b']$, where $a',a_i\in \rm{id}{\langle a \rangle }$ and $b',b_i\in \rm{id}{\langle b \rangle},, i=1,\dots,n$. Let us suppose that such non zero commutator is of the form
 $$
 c=[a_1,a_2,\dots,a_n; b_1,b_2,\dots,b_n;\omega].
 $$
 It is easy to check that
 $(\rm{id}{\langle a \rangle })^\mu=\rm{id}{\langle a^\mu \rangle}$. Hence, we can take $u_1,u_2,\dots,u_n\in \rm{id}_G{\langle u \rangle}$ and $v_1,v_2,\dots,v_n\in \rm{id}_G{\langle v \rangle}$ such that $u_i^\mu=a_i$ and $v_i^\mu=b_i, i=1,\dots,n$. It is clear that
 $$
 w= [u_1,u_2,\dots,u_n; v_1,v_2,\dots,v_n;\omega]\in [\rm{id}{\langle u \rangle},\rm{id}{\langle v \rangle}]
$$
and $w^\mu=c\neq 0$.
 We have
 $$
 [\rm{id}{\langle u \rangle },\rm{id}{{\langle v \rangle}}]\subseteq[T_1,T_2]\subseteq T_1\cap T_2.
 $$
 Finally, we have that $w\in T_1\cap T_2$ and  $w^\mu\neq 0$. Hence, $\mu\notin (T_1\cap T_2)'$ as desired.

 Now we prove the sufficiency. Suppose that $H$ is an equational domain. Assume to the contrary that $H$ is not a domain.  As a consequence, there exist two non-zero elements $a,b$ in $H$ such that
\begin{equation}\label{zd}
[\rm{id}{\langle a \rangle },\rm{id}{{\langle b \rangle}}]=0
\end{equation}
Let $X=\{x,y\}$. Take the free $\Omega$-group $F(X)$. The affine space $V={\rm Hom}(F(X), H)$ over $H$ can be identified with the set $H\times H$. In fact, every point $(h_1,h_2)\in H\times H$ determines the homomorphism $\mu:F(X)\rightarrow H$ such that $\mu(x)=h_1$ and $\mu(y)=h_2$ and vice versa.

 Consider two subvarieties $A$ and $B$ in $V$ defined by the equations $x=0$ and $y=0$, respectively. Since $H$ is an equational domain, $D=A\cup B$ is a subvariety of $V$. Let $f(x,y)$ be an element in $F(X)$ that belongs to $D'$. Then $f(a,0)=f(0,b)=0$. Let $Z=\{X_1,Y_1\}$ be a set in two indeterminates $X_1$ and $Y_1$ and  $F_\Upsilon(Z)$ be the free algebra in the variety $\Upsilon$ of all $\Omega$-groups generated by $Z$. Denote by $f(X_1,Y_1)$ the polynomial in $F_\Upsilon(Z)$ corresponding to $f(x,y)$. It is clear that the values of the polynomials $f(x,y)$ and $f(X_1,Y_1)$ in $H$ are equal. Since the formula (\ref{zd}) is valid, by Lemma \ref{com} we get
\begin{equation}\label{fin}
 f(a,b)=f( a, 0)+f(0,b)=0
\end{equation}
 From (\ref{fin}) we obtain that the point $(a,b)\in D''$. Since $D$ is the algebraic variety, $(a,b)\in D$. On the other hand, since $a$ and $b$ are nonzero elements in $H$, $(a,b)\notin D$. We have arrived at a contradiction. This ends the proof.
\end{proof}
As a consequence, Theorem \ref{main} holds for the variety of all linear algebras, for the variety of all groups, and for the variety of all modules (see \cite{P}, \cite{BPP}, \cite{DMR}). Theorem \ref{main} is also true for the so-called $CD$-variety of $\Omega$-groups (see Theorem 2 in \cite {P}).
\begin{definition}[\cite{L}]
An $\Omega$-group $H$ is called C-anticommutative (completely anticommutative) if each of its nonzero $\Omega$-subgroup is anticommutative.
\end{definition}
 It turns out that the concept of an equational domain and C-anticommutativity are closely related to each other.

In what follows, we use the following Proposition 4 from \cite{P}.
\begin{proposition}\label{P}
An $\Omega$-group $H$ is a domain if and only if $H$ is anticommutative.
\end{proposition}
We now give a useful criterion for an $\Omega$-group $H$ in a variety $\Theta$ to be an equational domain in this variety.
 \begin{proposition}\label{antic}
Every non-trivial $\Omega$-group $H$ in $\Theta$ is an equational domain in $\Theta$ if and only if $H$ is C-anticommutative.
\end{proposition}
 \begin{proof}
 Suppose that the $\Omega$-group $H$ is C-anticommutative. Let us take two nonzero elements $a$ and $b$ in $H$.
 Denote by $P$ the $\Omega$-subgroup of $H$ generated by the elements $a$ and $b$. Since $H$ is C-anticommutative, $P$ is anticommutative.  By Proposition \ref{P}, $P$ is without zero-divisors, i.e.,
  $[\rm{id}_P{\langle a \rangle},\rm{id}_P{\langle b\rangle}]\neq 0$. From this, it follows that
$[{\langle a \rangle},{\langle b\rangle}]\neq 0$. Therefore,
$[\rm{id}{\langle a \rangle}, \rm{id}{\langle b\rangle}]\neq 0$, i.e., $H$ is a domain. By Theorem \ref{main}, $H$ is an equational domain in $\Theta$.

 Now suppose that the $\Omega$-group $H$ is an equational domain. Let us show that it is C-anticommutative. Assume to the contrary, that $H$ is not $C$-anticommutative. Therefore, there exists an $\Omega$-subgroup $H_1$ of $H$ which is not anticommutative. Hence, there exist two non-zero elements $a,b\in H_1$ such that
\begin{equation}\label{eq_3}
[\rm{id}_{H_1}{\langle a \rangle},\rm{id}_{H_1}{\langle b \rangle}]=0
\end{equation}
 Take the free $\Omega$-group $F=F(X)$ generated by $X=\{x,y\}$. Denote by $V={\rm Hom}(F(X), H)$ the affine space over $H$.
 Consider two subvarieties $A$ and $B$ in $V$ defined by the equations $x=0$ and $y=0$, respectively. Since the $\Omega$-group $H$ is an equational domain, $D=A\cup B$ is a subvariety of $V$. Denote by $V_1={\rm Hom}(F(X), H_1)$ the subvariety of $V$ in the induced Zariski topology. Let $A_1$ and $B_1$ be the subvarieties of $V_1$ defined by the equations $x=0$ and $y=0$, respectively. Then we have
  $$
 A_1=V_1\bigcap A \,\, \mbox{and}\,\, B_1=V_1\bigcap B.
 $$
Since $D$ is subvarieties of $V$, $D_1=A_1\bigcup B_1$ is a subvariety of $V_1$. However, the same arguments given in the proof of the sufficiency of the conditions of Theorem \ref{main} show that $D_1$ is not a variety. We have a contradiction. This ends the proof.
\end{proof}

 \begin{corollary}\label{ass}
An $\Omega$-group $H$ in $\Theta$ is C-anticommutative if and only if for any non-zero elements $a$ and $b$ in $H$,
$[{\langle a \rangle},{\langle b \rangle}]\neq 0$.
\begin{proof}
Suppose that the elements $a$ and $b$ in $H$ satisfy the above condition. Therefore, $[\rm{id}{\langle a \rangle},\rm{id}{\langle b\rangle}]\neq 0$. Hence, $H$ is without zero divisors, i.e., $H$ is a domain. By Theorem \ref{main}, it is an equational domain. According to Proposition \ref{antic}, $H$ is C-anticommutative.

Let $H$ be C-anticommutative. Let $P$ be the $\Omega$-subgroup of $H$ generated by two non-zero elements $a$ and $b$ of $H$. Since $H$ is C-anticommutative, $P$ is anticommutative. By Proposition \ref{P}, $P$ is without zero-divisors, i.e.,
  $[\rm{id}_P{\langle a \rangle},\rm{id}_P{\langle b\rangle}]\neq 0$.  It follows that $[{\langle a \rangle},{\langle b\rangle}]\neq 0$ as desired.
\end{proof}
\end{corollary}
\begin{example}\label{ex_1}
 Now consider some examples of $\Omega$-groups that are not equational domains in varieties related to them.
 \begin{enumerate}
\item Every non-trivial module $M$ over an associative commutative ring $K$ with unit is not an equational domain in the variety of all modules over $K$, since it is abelian, i.e., $[M,M]=0$.

\item Every nontrivial soluble $\Omega$-group $G$ is not an equational domain in the variety $\Upsilon$ of all $\Omega$-groups.

Indeed, $G$ has a nontrivial abelian $\Omega$-subgroup. As a consequence, $G$ is not C-anticommutative.

The following examples were considered earlier in \cite{DMR}. However, using Proposition \ref{antic} and Corollary \ref{ass}, in contrast to the paper \cite{DMR}, we prove all statements (3)-(5) in a unified way.

 \item Every non-trivial group $G$ is not an equational domain in the variety of all groups.

  Indeed, $G$ contains a non-trivial cyclic group which is not anticommutative. Therefore, $G$ is not C-anticommutative.  By Proposition \ref{antic}, $G$ is not an equational domain.

\item Every non-trivial Lie algebra $L$ over an associative commutative ring $K$ is not an equational domain in the variety of all Lie algebras.

    In fact, $L$ contains a non-trivial cyclic subalgebra which is also not anticommutative. By Proposition \ref{antic}, $L$ is not an equational domain.

Denote by $L_{ring}=\{+,-,\cdot,0\}$ the language of associative rings. In the formulation of the following assertion, we use the language $L_{ring}$.

 \item An associative ring $A$ is an equational domain if and only if $A$ satisfies the formula:
  \begin{equation}\label{div}
   \forall x,y ((xy=yx=0)\Rightarrow [(x=0) \vee (y=0)]).
   \end{equation}
   \end{enumerate}

   Indeed, suppose that $A$ is an equational domain. By Proposition \ref{antic}, $A$ is C-anticommutative. Let $ab=ba=0$ for some elements $a$ and $b$ in $A$. From this it follows that
$$
   [{\langle a \rangle},{\langle b \rangle}]= 0.
$$
By Corollary \ref{ass}, $a=0$  or $b=0$.

   Conversely, assume that formula (\ref{div}) is true in $A$. Then for every non-zero elements $a$ and $b$ in $A$,
   $[{\langle a \rangle},{\langle b \rangle}]\neq 0$. By Corollary \ref{ass}, $A$ is C-anticommutative. According to Proposition \ref{antic}, $A$ is an equational domain.
   \end{example}

 \section{Acknowledgments}
I would like to thank Prof. B. I. Plotkin and Prof. E. B. Plotkin for useful suggestions and comments on this paper.

\bibliographystyle{amsalpha}

\begin{thebibliography}{A}

\bibitem [BMR] {BMR} G. Baumslag, A. Myasnikov, V. Remeslennikov, \textit {Algebraic geometry over groups I: Algebraic sets an ideal theory}, J. Algebra, vol. 219, 1999, pp. 16-19.

\bibitem [BPP] {BPP} A. Berzins, B. Plotkin, E. Plotkin,  \textit {Algebraic geometry in varieties of algebras with the given algebra of constants}, J. Math. Sci., 3, 2000, pp. 4039-4070.

\bibitem [DMR] {DMR} E. Daniyarova, A. Myasnikov, V, Remeslennikov, \textit {Algebraic geometry  over algebraic structures. IV. Equational domain and codomains}, Algebra and Logic, 6, vol 49, 2011, pp. 483-508.

\bibitem [H] {H} P. Higgins, \textit {Group with multiple operator}, Proc. London Math. Soc., 3, 1957, pp.366-416.

\bibitem [L] {L} R. Lipyanski, \textit {On stable $\Omega$-groups}, Proc. Latv. Acad. Sci. Sect. B Nat. Exact Appl. Sci. , 57, 4,
2003, pp. 102-105.

\bibitem [P]{P} B.I. Plotkin, \textit {Zero-divisor in group-based algebras. Algebras without zero divisors},  Buletinul A.S.
a R.M., 2, 1999, pp. 67-84.




 \end{thebibliography}

\end{document}